\documentclass[a4paper]{article}
\usepackage{amsmath,amsfonts,amssymb,amsthm}
\usepackage[english]{babel}
\usepackage{color}
\usepackage{tikz}

\newcommand{\cA}{{\mathcal A}}

\newcommand{\myspace}{\qquad\qquad\qquad}

\newtheorem{theorem}{Theorem}[section]
\newtheorem{lemma}[theorem]{Lemma}
\newtheorem{proposition}[theorem]{Proposition}

\numberwithin{equation}{section}

\begin{document}

\title{\textbf{{\Large Spectral analysis and rational decay rates of strong
solutions to a fluid-structure PDE system\thanks{%
The research of G.~Avalos was partially supported by the NSF~Grants
DMS-0908476 and DMS-1211232. The research of F.~Bucci was partially
supported by the Italian MIUR under the PRIN~2009KNZ5FK Project (\emph{%
Metodi di viscosit\`a, geometrici e di controllo per modelli diffusivi
nonlineari}), by the GDRE (Groupement De Recherche Europ\'een) CONEDP (\emph{%
Control of PDEs}), and also by the Universit\`a degli Studi di Firenze under
the Project \emph{Calcolo delle variazioni e teoria del controllo}.} }}}
\author{ George Avalos \\
%EndAName
{\normalsize {University of Nebraska-Lincoln}}\\
{\small {Lincoln NE, U.S.A.}}\\
{\small \texttt{gavalos@math.unl.edu}} \and Francesca Bucci \\
%EndAName
{\small {Universit\`a degli Studi di Firenze}}\\
{\small {Firenze, ITALY}}\\
{\small \texttt{francesca.bucci@unifi.it} }}

\date{}
\maketitle

\begin{abstract}
In this paper, we consider the problem of obtaining rational decay for a
particular time-evolving fluid-structure model, the type of which has been
considered in Chueshov and Ryzhkova (2013). 
In particular, this partial differential equation (PDE) system is composed of 
a three-dimensional Stokes flow which evolves within a three dimensional cavity. 
Moreover, on a (fixed) portion of the cavity wall, $\Omega$ say, a fourth order 
plate equation is invoked so as to describe the displacements along $\Omega$. 
Contact between these respective fluid and structure dynamics is established through 
the boundary interface $\Omega$. 
Our main result of decay is as follows: The PDE solutions of this fluid-structure PDE, corresponding to smooth initial data, decay at the rate of $\mathcal{O}(1/t)$. 
Our method of proof hinges upon the appropriate invocation of a relatively recent resolvent criterion for rational decays for linear $C_{0}$-semigroups.
\end{abstract}

% Top matter

%\date{\empty}

\section{Introduction} 

\subsection{The mathematical model: functional setting, main result}
In this paper we focus on the problem of deriving rational rates of uniform
decay for a fluid-structure partial differential equation (PDE) system; this
model has appeared repeatedly in the literature, in one form or another. 
(See e.g., \cite{chambolle}, \cite{igor}, \cite{avalos-clark}.) 
The composite systems of PDE describes the interactions of a viscous, incompressible fluid 
within a three-dimensional bounded domain ${\mathcal{O}}$ (the cavity) with an elastic 
dynamics along boundary interface $\Omega$. 
More precisely, let the walled cavity within which the fluid evolves be denoted as 
$\mathcal{O}$, a bounded subset of $\mathbb{R}^3$. This bounded set $\mathcal{O}$ will 
have sufficiently smooth boundary $\partial{\mathcal{O}}$, with 
$\partial \mathcal{O}=\overline{\Omega }\cup \overline{S}$, and $\Omega \cap S=\emptyset$. 
In particular, $\partial \mathcal{O}$ has the following specific spatial configuration: 
\begin{equation*}
\Omega \subset \left\{ x=(x_{1,}x_{2},0)\right\} \,,\quad S\subset \left\{
x=(x_{1,}x_{2},x_{3}):x_{3}\leq 0\right\} \,;
\end{equation*}
see, e.g., the picture below.

\smallskip % Picture

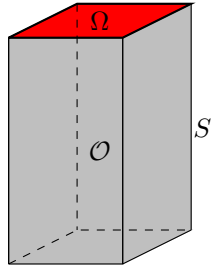
\begin{center}
\begin{tikzpicture}[scale=1.5]

\draw [fill=lightgray,lightgray] (.6,.3) rectangle (1.6,2.3);

\draw [fill=lightgray,lightgray] (0,0) rectangle (1,2);

\draw[fill=lightgray,lightgray] (0,0)--(1,0)--(1.6,.3)--(.6,.3);

\draw [thick, fill=red] (0,2)--(1,2)--(1.6,2.3)--(.6,2.3)--(0,2);

\node at (.8,1) {$\mathcal{O}$};

\node at (1.7,1.2) {$S$};

\node at (.8,2.15) {$\Omega$};

\draw  (0,0)--(0,2)--(1,2)--(1,0)--(0,0);

\draw (0,2)--(.6,2.3)--(1.6,2.3)--(1,2);

\draw (1.6,.3)--(1,0);

\draw [dashed] (0,0)--(.6,.3)--(1.6,.3);

\draw (0,0)--(0,2);

\draw (1,0)--(1,2);

\draw [dashed](.6,2.3)--(.6,.3);

\draw (1.6,.3)--(1.6,2.3);

\end{tikzpicture}
\end{center}

In consequence, if $\nu (x)$ denotes the exterior unit normal vector to $\partial \mathcal{O}$, then 
\begin{equation}
\nu |_{\Omega }=[0,0,1]\,.  \label{normal}
\end{equation}
With respect then to this geometry and with ``rotational inertia parameter'' $\rho \geq 0$, 
the PDE model is as follows, in solution variables $w(x,t)$ and 
$u(x,t)=[u^{1}(x,t),u^{2}(x,t),u^{3}(x,t)]$: 
\begin{subequations} \label{e:pde-model}
\begin{align}
& w_{tt}-\rho \Delta w_{tt}+\Delta ^{2}w=p|_{\Omega } & & \text{in}\;\Omega
\times (0,T)  \label{1} \\
& w=\frac{\partial w}{\partial \nu }=0 & & \text{on}\;\partial \Omega
\label{2} \\
& u_{t}-\Delta u+\nabla p=0 & & \text{in}\;\mathcal{O}\times (0,T)  \label{3}
\\
& \mathrm{div}(u)=0 & & \text{in}\;\mathcal{O}\times (0,T)  \label{4} \\
& u=0\;\; & & \text{on }\,S  \label{5-s} \\
& u=[u^{1},u^{2},u^{3}]=[0,0,w_{t}] & & \text{on }\,\Omega \,,
\label{5-omega}
\end{align}%
with initial conditions 
\end{subequations}
\begin{equation}
\lbrack w(0),w_{t}(0),u(0)]=[w_{0},w_{1},u_{0}]\in \mathbf{H}_{\rho }\,.
\label{ic}
\end{equation}%
Here, the space of initial data $\mathbf{H}_{\rho }$ is defined as follows:
Let 
\begin{equation}
\mathcal{H}_{\mathrm{f}}=\Big\{f\in \mathbf{L}^{2}(\mathcal{O}):\mathrm{div}%
(f)=0\,;\;f\cdot \nu |_{S}=0\Big\}\,;  \label{H_f}
\end{equation}%
and 
\begin{equation*}
V_{\rho }=%
\begin{cases}
L^{2}(\Omega )/\mathbb{R} & \text{if }\rho =0 \\[1mm] 
H_{0}^{1}(\Omega )\cap L^{2}(\Omega )/\mathbb{R} & \text{if }\rho >0\,.%
\end{cases}%
\end{equation*}%
Therewith, we then set 
\begin{equation}
\begin{split}
\mathbf{H}_{\rho }& =\Big\{\big[\omega _{0},\omega _{1},f\big]\in \big[%
H_{0}^{2}(\Omega )\cap L^{2}(\Omega )/\mathbb{R}\big]\times V_{\rho }\times 
\mathcal{H}_{\mathrm{f}}\,, \\[1mm]
& \qquad \qquad \qquad \text{with }\,f\cdot \nu |_{\Omega }=[0,0,f^{3}]\cdot
\lbrack 0,0,1]=\omega _{1}\Big\}\,.
\end{split}
\label{energy}
\end{equation}

As thus presented, the fluid PDE component of this fluid-structure dynamics
consists of a three dimensional incompressible Stokes flow which evolves
within the walled cavity $\mathcal{O}$, in solutions variables $u(x,t)$ and $p(x,t)$, 
with $u$ being the fluid velocity and $p$ the pressure contraint (see (\ref{3})-(\ref{4})). 
As for the structural component: on the cavity wall portion $\Omega$ a fourth order plate 
equation of either Kirchhoff ($\rho >0$) or Euler-Bernoulli ($\rho =0$) type is invoked to describe the displacements along $\Omega $; clamped boundary conditions are in place on 
$\partial \Omega$ (see (\ref{1})-(\ref{2})).

In addition, we note that for the fluid PDE component, the no-slip boundary
condition is in play \emph{only} on the wall $S$ of the fluid container (see
(\ref{5-s})). In particular, there is a matching of velocities on $\Omega $,
by way of accomplishing the coupling betweeen the respective fluid and
structure components; see \eqref{5-omega}. 
Moreover, the disparate dynamics are coupled via the Dirichlet boundary trace of the pressure; 
in particular, pressure variable $p$ appears as a forcing term in the $\Omega $-plate
equation (\ref{1}). We should also state that in general, fluid-structure
PDE models with ``fixed boundary interface'' $\Omega$ are physically relevant when
operating under the assumption that these cavity wall displacements are
small relative to the scale of the geometry; see \cite{du-et-al}.

\smallskip
If one performs a simple energy method, which would commence, by multiplying
structural PDE (\ref{1}) by $w_{t}$ and fluid PDE (\ref{3}) by $u$, and
subsequently integrate in time and space, one would find an underlying
dissipation of energy which governs the fluid-structure system. This
dissipation comes solely from the gradient of the fluid component $u$. 
Given this fluid dissipation which propagates onto the entire fluid-structure PDE,
an investigation here of stability properties for this coupled system would
seem to be appropriate.

\medskip
We proceed to write down an abstract realization of the fluid structure PDE 
\eqref{e:pde-model}-\eqref{ic}. To this end, let $A_{D}:{\mathcal{D}}(A_{D})\subset
L^{2}(\Omega )\rightarrow L^{2}(\Omega )$ be given by 
\begin{equation}
A_{D}g=-\Delta g\,,\quad D(A_{D})=H^{2}(\Omega )\cap H_{0}^{1}(\Omega )\,.
\label{dirichlet}
\end{equation}%
If we subsequently make the denotation for all $\rho \geq 0$, 
\begin{equation}
P_{\rho }=I+\rho A_{D}\,,\quad {\mathcal{D}}(P_{\rho })=%
\begin{cases}
L^{2}(\Omega ) & \text{if }\,\rho =0 \\ 
D(A_{D}) & \text{if }\,\rho >0%
\end{cases}%
\;,  \label{P}
\end{equation}%
then the mechanical PDE component \eqref{1}-\eqref{2} can be written as 
\begin{equation*}
P_{\rho }w_{tt}+\Delta ^{2}w=p|_{\Omega }\quad \text{ on}\;(0,T)\,.
\end{equation*}

Using that 
\begin{equation*}
{\mathcal{D}}(P_{\rho }^{1/2})=%
\begin{cases}
L^{2}(\Omega ) & \text{if }\;\rho =0 \\ 
H_{0}^{1}(\Omega ) & \text{if }\;\rho >0%
\end{cases}%
\;,
\end{equation*}%
(see \cite{grisvard}), then we can endow the Hilbert space $\mathbf{H}_{\rho
}$ with norm-inducing inner product 
\begin{equation*}
\big(\lbrack \omega _{0},\omega _{1},f],\big[\tilde{\omega}_{0},\tilde{\omega%
}_{1},\tilde{f}\big]\big)_{\mathbf{H}_{\rho }}=(\Delta \omega _{0},\Delta 
\tilde{\omega}_{0})_{\Omega }+(P_{\rho }^{1/2}\omega _{1},P_{\rho }^{1/2}%
\tilde{\omega}_{1})_{\Omega }+(f,\tilde{f})_{{\mathcal{O}}}\,,
\end{equation*}%
where $(\cdot ,\cdot )_{\Omega }$ and $(\cdot ,\cdot )_{\mathcal{O}}$ are
the $L^{2}$-inner products on their respective geometries.

We note here, as it was in \cite{igor}, the necessity for imposing that wave
initial displacement and velocity each have zero mean average. \newline
To see this: invoking the boundary condition \eqref{5-s}-\eqref{5-omega} and
the fact that the normal vector $\nu$ coincides with $[0,0,1]$ on $\Omega$,
we have then by Green's formula, that for all $t\geq 0$, 
\begin{equation}  \label{zero}
\int_{\Omega} w_t(t)\, dx =\int_{\Omega }u^3(t)\,dx = \int_{\partial{%
\mathcal{O}}} u(t)\cdot \nu\, d\sigma=0\,.
\end{equation}
And so we have necessarily, 
\begin{equation*}
\int_{\Omega} w(t)\, dx =\int_{\Omega} w_{0}\,dx \quad \text{for all }\,
t\ge 0\,.
\end{equation*}
This accounts for the choice of the structural finite energy space
components for $\mathbf{H}_{\rho}$, in \eqref{energy}.

\medskip Well-posedness of the initial/boundary value problem %
\eqref{e:pde-model}-\eqref{ic} has been fully discussed in \cite%
{avalos-clark} for both cases $\rho >0$ and $\rho =0$. The proof of
well-posedness provided therein hinges upon demonstrating the existence of a
modeling semigroup $\big\{e^{{\mathcal{A}}_{\rho }t}\big\}_{t\geq 0}\subset {%
\mathcal{L}}(\mathbf{H}_{\rho })$, for appropriate generator ${\mathcal{A}}:%
\mathbf{H}_{\rho }\rightarrow \mathbf{H}_{\rho }$. Subsequently, by means of
this family, the solution to \eqref{e:pde-model}-\eqref{ic}, for initial
data $[w_{0},w_{1},u_{0}]\in \mathbf{H}_{\rho }$, will then of course be
given via the relation 
\begin{equation}
\left[ 
\begin{array}{c}
w(t) \\ 
w_{t}(t) \\ 
u(t)%
\end{array}%
\right] =e^{\mathcal{A}_{\rho }t}\left[ 
\begin{array}{c}
w_{0} \\ 
w_{1} \\ 
u_{0}%
\end{array}%
\right] \in C([0,T];\mathbf{H}_{\rho })\,.  \label{semi}
\end{equation}%
We recall here that the particular choice here of generator ${\mathcal{A}}%
_{\rho }:\mathbf{H}_{\rho }\rightarrow \mathbf{H}_{\rho }$ is dictated by
the following consideration, whose proof is given for the reader's
convenience. 

\begin{lemma}
If $p(t)$ is a viable pressure variable for \eqref{e:pde-model}-\eqref{ic},
then pointwise in time $p(t)$ necessarily satisfies the following boundary
value problem, for $[w(t),u(t)]$ \textquotedblleft smooth
enough\textquotedblright : 
\begin{equation}
\begin{cases}
\Delta p=0 & \text{in }\;\mathcal{O} \\[1mm] 
\frac{\partial p}{\partial \nu }=\Delta u\cdot \nu \big|_{S} & \text{on }\;S
\\[1mm] 
\frac{\partial p}{\partial \nu }+P_{\rho }^{-1}p=P_{\rho }^{-1}\Delta
^{2}w+\Delta u^{3}\big|_{\Omega } & \text{on }\;\Omega \,.%
\end{cases}
\label{bvp}
\end{equation}
\end{lemma}

\begin{proof}
To show that $p$ is harmonic in $\Omega$, we take the divergence of both sides of \eqref{3} 
and use the divergence free condition in (\ref{4}). 
Moreover, dotting both sides of (\ref{3}) with the unit normal vector $\nu $, and then subsequently taking the resulting trace on $S$ will yield the boundary condition in \eqref{bvp} 
that pertains to $S$. 
(Implicitly, we are also using the fact that $u=0$ on $S$.)

Finally, we consider the particular geometry which is in play (where $\nu=(0,0,1)$ on $\Omega$). 
Using the equation (\ref{1}) and the boundary condition \eqref{5-omega}, 
we have on $\Omega$ 
\begin{equation*}
\begin{split}
P_{\rho }^{-1}\Delta ^{2}w & =-w_{tt}+ P_{\rho }^{-1}p\big|_{\Omega } 
\\
\myspace &=-\frac{d}{dt}(0,0,w_{t})\cdot \nu +P_{\rho }^{-1}p\big|_{\Omega } 
\\
& =-\left[ u_{t}\cdot \nu \right]_{\Omega }+ P_{\rho}^{-1}p\big|_{\Omega } 
\\
&=-\left[ \Delta u\cdot \nu \right] _{\Omega }+\frac{\partial p}{\partial \nu }\Big|_{\Omega }
+ P_{\rho }^{-1}p\big|_{\Omega }\,,
\end{split}
\end{equation*}
which gives the boundary condition in \eqref{bvp} 
that pertains to $\Omega$.
\end{proof}

The boundary value problem (BVP) \eqref{bvp} can be solved through the
agency of appropriate harmonic extensions from the boundary of ${\mathcal{O}}
$, that are the ``Robin-Neumann'' maps $R_{\rho}$ and $\tilde{R}_{\rho }$
defined by 
\begin{equation*}
\begin{aligned} R_{\rho }g &= f\Longleftrightarrow \Big\{ \Delta f=0\text{ \
in }\mathcal{O}\,, \; \frac{\partial f}{\partial \nu }+P_{\rho
}^{-1}f=g\text{ \ on } \Omega\,, \; \frac{\partial f}{\partial \nu
}=0\text{\ on }S\Big\}\,; \\ \tilde{R}_{\rho }g &=f\Longleftrightarrow
\Big\{ \Delta f=0\text{ \ in } \mathcal{O}\,, \; \frac{\partial f}{\partial
\nu }+P_{\rho }^{-1}f=0\text{\ on }\Omega\,, \; \frac{\partial f}{\partial
\nu }=g\text{ on } S\Big\}\,. \end{aligned}
\end{equation*}
It is well known that for all real $s$, 
\begin{equation}  \label{Rs}
R_{\rho }\in \mathcal{L}\big(H^{s}(\Omega),H^{s+3/2}(\mathcal{O}) \big)\text{%
; \ }\tilde{R}_{\rho }\in \mathcal{L}\big(H^{s}(S),H^{s+3/2}(\mathcal{O})%
\big)\,;
\end{equation}
see, e.g., \cite{L-M}. (We are also using implicity the fact that $%
P_{\rho}^{-1}$ is positive definite, self-adjoint on $\Omega $.)

\smallskip Therewith, the pressure variable $p(t)$, as necessarily the
solution of \eqref{bvp}, can be written pointwise in time as 
\begin{equation}
p(t)=G_{\rho,1}w(t)+G_{\rho,2}u(t),  \label{p}
\end{equation}
where $G_{\rho,1}$ and $G_{\rho,2}$ % ---$G_1$ and $G_2$, in short---
are defined as follow: 
\begin{align}
& G_{\rho,1}w=R_{\rho }(P_{\rho }^{-1}\Delta ^{2}w)\,;  \label{G1} \\
& G_{\rho,2}u =R_{\rho }(\left. \Delta u^{3}\right\vert _{\Omega })+ \tilde{R%
}_{\rho }(\left. \Delta u\cdot \nu \right\vert_{S})\,.  \label{G2}
\end{align}

\smallskip These relations suggest the following choice for the generator $%
\mathcal{A}_{\rho }:\mathbf{H}_{\rho }\rightarrow \mathbf{H}_{\rho}$. We set 
\begin{equation}  \label{domain}
{\mathcal{A}}_\rho \equiv 
\begin{bmatrix}
0 & I & 0 \\ 
-P_{\rho }^{-1}\Delta ^{2}+P_{\rho }^{-1}G_{\rho ,1}\big|_{\Omega } & 0 & 
P_{\rho }^{-1}G_{\rho ,2}\big|_{\Omega } \\ 
-\nabla G_{\rho ,1} & 0 & \Delta -\nabla G_{\rho ,2}%
\end{bmatrix}%
\end{equation}
with domain 
\begin{equation}  \label{e:domain-of-generator}
\begin{split}
{\mathcal{D}}({\mathcal{A}}_\rho)&=\Big\{ \big[w_1,w_2,u\big] \in \mathbf{H}%
_{\rho }: w_1\in H^{3}(\Omega )\cap H_{0}^{2}(\Omega)\,, \; w_2\in
H_{0}^{2}(\Omega )\,, \; u\in \mathbf{H}^{2}(\mathcal{O})\,; \\[1mm]
& \qquad u=0 \; \text{on}\; S\,, \quad u=(0,0,w_2) \; \text{on} \;\Omega\,,
\\[1mm]
& \qquad \Delta u\cdot \nu\big|\in H^{-1/2}(\Omega)\,, \; \text{and hence}
\; G_{\rho,1}w_1+G_{\rho,2}u\in H^1({\mathcal{O}}) \Big\}\,.
\end{split}%
\end{equation}
assuming $\rho$ is positive, whereas more precisely the first membership in %
\eqref{e:domain-of-generator} is as follows: 
\begin{equation*}
w_1\in {\mathcal{S}}_{\rho }:= 
\begin{cases}
H^{4}(\Omega )\cap H_{0}^{2}(\Omega ) & \rho =0 \\ 
H^{3}(\Omega )\cap H_{0}^{2}(\Omega) & \rho >0\,.%
\end{cases}%
\end{equation*}

\medskip 
Thus, we remind the reader that well-posedness for the dynamics
governed by the operator ${\cA}_\rho$, when $\rho =0$ (i.e., when
the elastic equation is the Euler-Bernoulli one), was originally established
in \cite{igor}, by using Galerkin approximations. 
A novel proof of well-posedness pertaining to both cases $\rho =0$ and $\rho >0$, based upon
the classical Lumer-Phillips Theorem as well as on a clever use of the 
Babu\v{s}ka-Brezzi~Theorem (see, e.g., \cite[p.~116]{kesavan}) has been
recently given in \cite{avalos-clark}. 
The corresponding statement is as follows.

\begin{theorem}[\textbf{Well-posedness \protect\cite{avalos-clark}}] \label{well} 
The operator $\mathcal{A}_{\rho }:\mathbf{H}_{\rho }\rightarrow 
\mathbf{H}_{\rho }$ defined by \eqref{domain}-\eqref{e:domain-of-generator}
generates a $C_{0}$-semigroup of contractions $\left\{e^{\mathcal{A}%
t}\right\}_{t\geq 0}$ on $\mathbf{H}_{\rho }$. Thus, given $%
[w_{0},w_{1},u_{0}]\in \mathbf{H}_{\rho }$, the weak solution $%
[w,w_{t},u]\in C([0,T];\mathbf{H}_{\rho })$ of (\ref{e:pde-model})-(\ref{ic}) is given by (\ref{semi}).
\end{theorem}

\smallskip 
In the present work the long-time behaviour, as $t\rightarrow +\infty $, of the linear dynamics described by \eqref{e:pde-model} is addressed, with focus on the more challenging case $\rho >0$ (the elastic equation is the Kirchhoff one). 
When $\rho =0$, uniform (exponential) stability of finite energy solutions holds true; this issue has been dealt with in \cite{igor}, by using Lyapunov function arguments (in the \emph{time
domain}). 
A different proof of the aforesaid result has been subsequently given in \cite{avalos-bucci-cortona}, with a proof geared rather toward establishing the necessary resolvent estimates in the \emph{frequency domain}. 
The very same {\em `frequency domain perspective'} enables us to infer that in the case 
$\rho >0$, a weaker notion of uniform decay will prevail for the fluid-structure 
PDE \eqref{e:pde-model}-\eqref{ic}. 
In particular, the main result of this paper is the following
stability result pertaining to \emph{strong} solutions, which provides sharp 
\emph{polynomial} rates of decay.

% MAIN RESULT

\begin{theorem}[\textbf{Main result: Rational decay rates}]
\label{t:main} Let the rotational inertia parameter $\rho $ be positive in %
\eqref{1}. Then for initial data $[w_{0},w_{1},u_{0}]\in D(\mathcal{A}_{\rho
})$, the corresponding solution $[w,w_{t},u]\in C([0,T];D(\mathcal{A}_{\rho
}))$ of \eqref{e:pde-model}-\eqref{ic} satisfies the following decay rate
for time $t$ large enough: 
\begin{equation}
\Vert \lbrack w(t),w_{t}(t),u(t)]\Vert _{\mathbf{H}_{\rho }}\leq \frac{C}{t}%
\,\Vert \lbrack w_{0},w_{1},u_{0}]\Vert _{D({\mathcal{A}}_{\rho })}\,.
\label{main_est}
\end{equation}
\end{theorem}

In what follows, rotational parameter $\rho$ will be positive always.

\bigskip
\subsection{Background and further remarks}
We should note that in principle, one might attempt to derive the rational
decay estimate (\ref{main_est}) by an analysis in the ``time domain''; the 
associated energy method is in principle abstractly outlined in 
\cite[Theorem~3.2.2, p.~43]{lasiecka}. 
However, the details of proof in the time domain, at least from our vantage point, 
would seem to be quite daunting, if it can be done at all, in the $t$-domain. 
\\
(We note that the time domain approach---along energy methods, combined with interpolation 
techniques---underlines as well the recent work \cite{alabau-et-al}, which provides a novel criterion to derive the decay rates of the solutions to evolutionary PDE systems, whose 
range of application includes certain coupled PDE systems of {\em hyperbolic} type.)

Instead, we choose to operate here in the ``frequency domain'', by invoking an energy 
method with respect to a formally `Laplace transformed' version of the system 
\eqref{e:pde-model}-\eqref{ic}, with an ultimate view of invoking the sharp
resolvent criterion in \cite{borichev} (see a penultimate version of this
resolvent criterion in \cite{liu}). 
Such a {\em frequency domain} approach was previously invoked in \cite{avalos-trigg}, by way of establishing rational decays for a wholly different fluid-structure PDE model. 
We should state at the outset that one advantage which the frequency domain approach enjoys
over the time domain approach, is that the former eventually allows for a adequate treatment of the pressure variable (as it appears as a forcing term in the $\Omega$-plate equation). 
In particular, upon formally taking the Laplace Transform of \eqref{e:pde-model}-\eqref{ic}, 
so as to obtain a corresponding static fluid-structure system with frequency domain parameter 
$\beta$ (see \eqref{e:static-beta} below), one can then attempt to invoke classic
Stokes Theory for (static) incompressible fluid flows. 

Alternatively, if one were working directly with the time evolving fluid-structure system 
\eqref{e:pde-model}-\eqref{ic}, by way of analyzing the pressure term $p(x,t)|_\Omega$, 
it seems likely to us that there would be the necessity of microlocalizing the fluid-structure system in order to obtain the required \textit{a priori} estimates. 
Besides being quite technical in its own right, such a pseudo-differential approach might 
even be ultimately unavailing, inasmuch as there would be the issue of keeping a close track 
of the time dependent constants which would surely accumulate in the course of
such a microlocal analysis. 
Hence, we are drawn instead to a frequency domain approach which would ultimately invoke the resolvent criterion Theorem~\ref{t:borichev-tomilov} below. 

\medskip
We should also state that uniform stability results for higher dimensional coupled PDE models (namely, involving equations on $n$-dimensional manifolds, with $n>1$) which are attained via a {\em frequency domain} approach are largely not available in the literature; see e.g., 
\cite{avalos-trigg}, \cite{avalos-trigg02}.
We recall the recent polynomial decay result obtained in \cite{MG-1} for a complicated Mindlin-Timoshenko plate model, which also depends upon a frequency domain approach and an argument by contradiction, with a view of invoking the aforesaid resolvent criterion in \cite{borichev}; see also \cite{MG-2}.
In general, those few higher dimensional {\em frequency domain} results which are available typically invoke an argument by contradiction, in the style of \cite{liu}, by way of 
establishing the requisite resolvent estimate in Theorem~\ref{t:borichev-tomilov} below. 
By contrast, in the present paper we \emph{explicitly} generate the necessary frequency domain estimates; see also \cite{avalos-trigg} and \cite{avalos-trigg02}.

%\cite{munoz-et-al},  Timoshenko system with thermoelastic dissipation. 

% SPECTRAL ANALYSIS
\section{Spectral analysis}
In order to establish the sharp estimate of the decay rates for the solutions of the PDE system, 
we will use a recent and powerful \emph{frequency domain} criterion by A.~Borichev and Y.~Tomilov, which for the readers' convenience is recorded below.

\begin{theorem}[\protect\cite{borichev}]
\label{t:borichev-tomilov} %[Theorem~2.4] 
Let $(T(t))_{t\geq 0}$ be a bounded $C_{0}$-semigroup on a Hilbert space $H$
with generator $A$, such that $i\mathbb{R}\subset \rho (A)$. Then, for fixed 
$\alpha >0$ the following are equivalent: 
\begin{align}
(i)\quad & R(is;A)=O(|s|^{\alpha })\,,\qquad |s|\rightarrow \infty \,;
\label{e:asymptotic-tomilov} \\[1mm]
(ii)\quad & \Vert T(t)x\Vert _{H}=o(t^{-1/\alpha })\Vert x\Vert _{{\mathcal{D%
}}(A)}\,,\qquad t\rightarrow +\infty \,.  \notag
\end{align}
\end{theorem}

\medskip To apply the above result, we preliminary need to show that the
imaginary axis belongs to the resolvent set of the dynamics operator ${%
\mathcal{A}}_\rho$. The present Section is entirely devoted to this
objective.

\subsection{$\protect\lambda =0$ is in the resolvent set $\protect\rho ({%
\mathcal{A}}_{\protect\rho })$}

We begin our analysis by showing that the dynamics operator ${\mathcal{A}}%
_{\rho }$ is \emph{boundedly invertible} on the state space $\mathbf{H}$;
the corresponding statement is given separately. In this connection, we will
need the following trace regularity result, which is readily established;
see e.g., Proposition 2 of \cite{avalos-clark}.

\begin{proposition}
\label{extra}Suppose a function $\mu \in \mathbf{L}^{2}(\mathcal{O})$ and
pair $(\varrho ,h)\in H^{1}(\Omega )\times \mathbf{L}^{2}(\mathcal{O})$
satisfy the relation%
\begin{equation}
-\Delta \mu +\nabla \varrho =h\text{, }  \label{dot}
\end{equation}%
where $\mathrm{div}(\mu )=\mathrm{div}(h)=0$. Then $\left. \Delta \mu \cdot
\nu \right\vert _{\partial \mathcal{O}}\in \mathbf{H}^{-\frac{1}{2}%
}(\partial \mathcal{O})$, with the estimate%
\begin{equation}
\left\Vert \left. \Delta \mu \cdot \nu \right\vert _{\partial \mathcal{O}%
}\right\Vert _{\mathbf{H}^{-\frac{1}{2}}(\partial \mathcal{O})}\leq C\left[
\left\Vert \varrho \right\Vert _{H^{1}(\mathcal{O})}+\left\Vert h\right\Vert
_{\mathbf{L}^{2}(\mathcal{O})}\right] .  \label{est}
\end{equation}
\end{proposition}

\bigskip

\begin{proposition}
The generator $\mathcal{A}_{\rho }:D(\mathcal{A}_{\rho })\subset \mathbf{H}%
_{\rho }\rightarrow \mathbf{H}_{\rho }$ is boundedly invertible on $\mathbf{H%
}_{\rho }$. Namely, $\lambda =0$ is in the resolvent set of $\mathcal{A}%
_{\rho }$.
\end{proposition}

\begin{proof}
Given data $[\omega _{1}^{\ast },\omega _{2}^{\ast },\mu ^{\ast }]\in 
\mathbf{H}_{\rho }$, we look for $[\omega _{1},\omega _{2},\mu ]\in {%
\mathcal{D}}({\mathcal{A}}_{\rho })$ which solves 
\begin{equation}
\mathcal{A}_{\rho }\left[ 
\begin{array}{c}
\omega _{1} \\ 
\omega _{2} \\ 
\mu%
\end{array}%
\right] =\left[ 
\begin{array}{c}
\omega _{1}^{\ast } \\ 
\omega _{2}^{\ast } \\ 
\mu ^{\ast }%
\end{array}%
\right] \,.  \label{resolve}
\end{equation}%
To this end, we must search $[\omega _{1},\omega _{2},\mu ]$ in ${\mathcal{D}%
}({\mathcal{A}}_{\rho })$ and $\pi _{0}\in H^{1}({\mathcal{O}})$ which solve 
\begin{subequations}
\label{e:resolvent-system}
\begin{align}
& \omega _{2}=\omega _{1}^{\ast } & & \text{in $\Omega $}
\label{e:resolvent_r1} \\
& P_{\rho }^{-1}\Delta ^{2}\omega _{1}-P_{\rho }^{-1}\pi _{0}\big|_{\Omega
}=-\omega _{2}^{\ast } & & \text{in $\Omega $}  \label{e:resolvent_r2} \\
& \omega _{1}=\frac{\partial \omega _{1}}{\partial n}=0 & & \text{on $%
\partial \Omega $}  \label{e:resolvent_r3} \\
& \Delta \mu -\nabla \pi _{0}=\mu ^{\ast } & & \text{in ${\mathcal{O}}$}
\label{e:resolvent_r4} \\
& \text{div}\mu =0 & & \text{in ${\mathcal{O}}$}  \label{e:resolvent_r5} \\
& \mu =0 & & \text{on $S$}  \label{e:resolvent_r6} \\
& \mu =(0,0,\omega _{2}) & & \text{on $\Omega $}\,.  \label{e:resolvent_r7}
\end{align}%
Moreover, we must justify that the pressure variable $\pi _{0}$ is given by
the expression 
\end{subequations}
\begin{equation}
\pi _{0}=G_{1}\omega _{1}+G_{1}\mu \,,  \label{e:resolvent-pi_0}
\end{equation}%
where $G_{i}$ denote, in short, the linear operators $G_{\rho ,i}$, $i=1,2$
defined by \eqref{G1} and \eqref{G2}, respectively (in line with the
appearance of ${\mathcal{A}}_{\rho }$ in \eqref{domain}).

\medskip \noindent \emph{(i) The Plate Velocity. } From %
\eqref{e:resolvent_r1}, the velocity component $\omega _{2}$ is immediately
resolved.

\medskip \noindent \emph{(ii) The Fluid Velocity. } We next consider the
Stokes system \eqref{e:resolvent_r4}-\eqref{e:resolvent_r7}. From %
\eqref{e:resolvent_r1} and \eqref{e:resolvent_r7} it follows that $\mu
|_{\partial \mathcal{O}}$ satisfies 
\begin{equation}
\int_{\partial \mathcal{O}}\mu \cdot \nu \,d\sigma =\int_{\Omega }\big[%
0,0,\mu _{3}\big]^{T}\cdot \nu \,d\sigma =\int_{\Omega }\omega _{2}\,d\sigma
=\int_{\Omega }\omega _{1}^{\ast }\,d\sigma =0  \label{cond}
\end{equation}%
(as $[\omega _{1}^{\ast },\omega _{2}^{\ast },\mu ^{\ast }]\in \mathbf{H}%
_{\rho }$). Since this compatibility condition is satisfied and data $%
\left\{ \mu ^{\ast },\omega _{1}^{\ast }\right\} \in \mathbf{L}^{2}(\mathcal{%
O})\times H_{0}^{2}(\Omega )$, we can find a unique (fluid and pressure)
pair $\left( \mu ,q_{0}\right) \in \left[ \mathbf{H}^{2}(\mathcal{O})\cap 
\mathcal{H}_{f}\right] \times \mathbf{H}^{1}(\mathcal{O})/\mathbb{R}$ which
solves 
\begin{subequations}
\label{e:static-stokes}
\begin{eqnarray}
&&\Delta \mu -\nabla q_{0}=\mu ^{\ast }\ \text{in }\mathcal{O}
\label{static-stokes-s1} \\
&&\text{div}(\mu )=0\text{ \ in }\mathcal{O}  \label{static-stokes-s2} \\
&&\mu =(0,0,\omega _{1}^{\ast })\text{ \ on }\Omega \,,\;\text{ \ }\mu =0%
\text{ \ on }S\,.  \label{static-stokes-s3}
\end{eqnarray}%
Moreover, one has the estimate 
\end{subequations}
\begin{equation}
\left\Vert \mu \right\Vert _{\mathbf{H}^{2}(\mathcal{O})\cap \mathcal{H}%
_{f}}+\left\Vert q_{0}\right\Vert _{\mathbf{H}^{1}(\mathcal{O})/\mathbb{R}%
}\leq C\,\left[ \left\Vert \mu ^{\ast }\right\Vert _{\mathcal{H}%
_{f}}+\left\Vert \omega _{1}^{\ast }\right\Vert _{H_{0}^{2}(\Omega )}\,%
\right] ;  \label{stokes}
\end{equation}%
see e.g., \cite[Proposition 2.3, p.~25]{temam}.

\medskip \noindent \emph{(iii) The Mechanical Displacement. } Subsequently,
we consider the equations \eqref{e:resolvent_r2}-\eqref{e:resolvent_r3}
pertaining to the (plate) component $\omega _{1}$. By classical elliptic
theory % (see \cite{L-M}) 
there exists a solution $\widehat{\omega }_{1}\in H^{3}(\Omega )\cap
H_{0}^{2}(\Omega )$ to the boundary value problem 
\begin{equation}
\begin{cases}
\Delta ^{2}\widehat{\omega }_{1}=q_{0}|_{\Omega }-P_{\rho }\omega _{2}^{\ast
} & \text{in $\Omega $} \\[1mm] 
\widehat{\omega }_{1}=\frac{\partial \widehat{\omega }_{1}}{\partial \nu }=0
& \text{on $\partial \Omega $}\,,%
\end{cases}
\label{p_1}
\end{equation}%
where $q_{0}$ is the (pressure) variable in \eqref{static-stokes-s1};
moreover, the following estimate holds true:%
\begin{eqnarray}
\left\Vert \widehat{\omega }_{1}\right\Vert _{H^{3}(\Omega )\cap
H_{0}^{2}(\Omega )} &\leq &C\left\Vert \left. q_{0}\right\vert _{\Omega
}+P_{\rho }\omega _{2}^{\ast }\right\Vert _{H^{-1}(\Omega )}  \notag \\
&\leq &C\left\Vert \left. q_{0}\right\vert _{\Omega }\right\Vert
_{H^{1/2}(\Omega )}+\left\Vert P_{\rho }\omega _{2}^{\ast }\right\Vert
_{H^{-1}(\Omega )}  \notag \\
&\leq &C\left\Vert [\omega _{1}^{\ast },\omega _{2}^{\ast },\mu ^{\ast
}]\right\Vert _{\mathbf{H}_{\rho }}.  \label{plate_hat}
\end{eqnarray}%
(In the last inequality we have also invoked Sobolev trace theory and %
\eqref{stokes}).

Now if, as in \cite{igor}, we let $\mathbb{P}$ denote the orthogonal
projection of $H_{0}^{2}(\Omega )$ onto $H_{0}^{2}(\Omega )\cap L^{2}(\Omega
)/\mathbb{R}$ (orthogonal with respect to the inner product $[\omega ,\tilde{%
\omega}]\rightarrow \left( \Delta \omega ,\Delta \tilde{\omega}\right)
_{\Omega }$), then one can readily show that its orthogonal complement $I-%
\mathbb{P}$ can be characterized as 
\begin{equation}
(I-\mathbb{P})H_{0}^{2}(\Omega )=\text{span}\Big\{\varpi :\;\Delta
^{2}\varpi =1\text{ in }\Omega \,,\;\varpi =\frac{\partial \varpi }{\partial
\nu }=0\text{ on }\partial \Omega \Big\}\,;  \label{constant}
\end{equation}%
see \cite[Remark 2.1, p.~6]{igor}. \newline
With these projections, we then set 
\begin{equation}
\omega _{1}:=\mathbb{P}\widehat{\omega }_{1}\,,\qquad \pi _{0}:=q_{0}-\Delta
^{2}(I-\mathbb{P})\widehat{\omega }_{1}\,;  \label{assign}
\end{equation}%
therefore, by (\ref{p_1}) and $\widehat{\omega }_{1}=\mathbb{P}\widehat{%
\omega }_{1}+(I-\mathbb{P})\widehat{\omega }_{1}$, we will have that $\omega_1$ solves 
\eqref{e:resolvent_r2}-\eqref{e:resolvent_r3}. 
(And of course since $\pi _{0}$ and $q_{0}$ differ only by a constant, then the pair 
$(\mu,\pi_0)$ also solves \eqref{e:resolvent_r4}-\eqref{e:resolvent_r7}.) 
Thus, in view of elliptic theory, (\ref{stokes}) and (\ref{plate_hat}), we obtain the estimate 
\begin{equation}
\begin{split}
& \left\Vert \omega _{1}\right\Vert _{H^{3}(\Omega )\cap H_{0}^{2}(\Omega
)\cap L^{2}(\Omega )/\mathbb{R}}+\left\Vert \pi _{0}\right\Vert _{H^{1}(%
\mathcal{O})}\leq  \\[2mm]
& \qquad \qquad \qquad \leq C\left( \left\Vert \Delta ^{2}(I-\mathbb{P})%
\widehat{\omega }_{1}\right\Vert _{L^{2}(\Omega )}+\left\Vert
q_{0}\right\Vert _{H^{1}(\mathcal{O})/\mathbb{R}}+\left\Vert P_{\rho }\omega
_{2}^{\ast }\right\Vert _{H^{-1}(\Omega )}\right)  \\[2mm]
& \qquad \qquad \qquad \leq C\left\Vert [\omega _{1}^{\ast },\omega
_{2}^{\ast },\mu ^{\ast }]\right\Vert _{\mathbf{H}_{\rho }}\,,
\end{split}
\label{plate}
\end{equation}
where implicity we are also using the fact that $\Delta ^{2}(I-\mathbb{P})\in \mathcal{L}
(H_{0}^{2}(\Omega ),\mathbb{R})$, by the Closed Graph Theorem.

\medskip \noindent \emph{(v) Resolution of the Pressure. } 
At this point we invoke Proposition~\ref{extra} and \eqref{stokes} to have the following trace regularity for the fluid velocity in \eqref{e:static-stokes}: 
\begin{equation}
\begin{split}
\big\|\Delta \mu \cdot \nu \big\|_{H^{-1/2}(\partial \mathcal{O})}& \leq C\,%
\big[\Vert q_{0}\Vert _{H^{1}(\mathcal{O})}+\Vert \mu ^{\ast }\Vert _{%
\mathbf{L}^{2}(\mathcal{O})}\big] \\[1mm]
& \leq C\,\big[\Vert \mu ^{\ast }\Vert _{\mathcal{H}_{f}}+\Vert \omega
_{1}^{\ast }\Vert _{H_{0}^{2}(\Omega )}\big]\,.
\end{split}
\label{v_tr}
\end{equation}%
In consequence, the pressure variable $\pi _{0}$ of problem %
\eqref{e:resolvent_r1}-\eqref{e:resolvent_r7}---given explicitly in %
\eqref{assign}---solves \textit{a fortiori} 
\begin{equation}
\begin{cases}
\Delta \pi _{0}=0 & \text{in $\mathcal{O}$} \\[1mm]
\frac{\partial \pi _{0}}{\partial \nu }=\Delta \mu \cdot \nu |_{S} & \text{%
on $S$} \\[1mm]
\frac{\partial \pi _{0}}{\partial \nu }+P_{\rho }^{-1}\pi _{0}=P_{\rho
}^{-1}\Delta ^{2}\omega _{1}+\Delta \mu ^{3}\big|_{\Omega } & \text{on $%
\Omega $.}%
\end{cases}
\label{e:statisfied-by-pizero}
\end{equation}%
We justify the previous assertion. Applying the divergence operator to both
sides of \eqref{e:resolvent_r4} and using $\text{div}\mu =\text{div}\mu
^{\ast }=0$, we obtain that $\pi _{0}$ is harmonic in ${\mathcal{O}}$. Next,
dotting both sides of \eqref{e:resolvent_r4} with repect to the normal
vector, and subsequently taking the boundary trace on the portion $S$, we
get the corresponding boundary condition in \eqref{e:statisfied-by-pizero}.
(Implicitly we are also using $\mu ^{\ast }\cdot \nu |_{S}=0$, as $[\omega
_{1}^{\ast },\omega _{2}^{\ast },\mu ^{\ast }]\in \mathbf{H}_{\rho }$).

Finally, since $\mu ^{\ast }\cdot \nu |_{\Omega }=\omega _{2}^{\ast }$, as $%
[\omega _{1}^{\ast },\omega _{2}^{\ast },\mu ^{\ast }]\in \mathbf{H}_{\rho }$%
, from \eqref{e:resolvent_r2} it follows that 
\begin{eqnarray*}
P_{\rho }^{-1}\pi _{0}\big|_{\Omega } &=&\omega _{2}^{\ast }+P_{\rho
}^{-1}\Delta ^{2}\omega _{1} \\
&=&\Delta \mu \cdot \nu |_{\Omega }-\nabla \pi _{0}\cdot \nu |_{\Omega
}+P_{\rho }^{-1}\Delta ^{2}\omega _{1},
\end{eqnarray*}%
which gives the boundary condition on $\Omega $. \newline
Necessarily then, the pressure term must be given by the expression 
\begin{equation}
\pi _{0}=G_{1}\omega _{1}+G_{2}\mu \in H^{1}({\mathcal{O}})\,,  \label{d}
\end{equation}%
with the well-definition of the right hand side ensured by \eqref{v_tr}.

\medskip Finally, we collect: the fluid variable $\mu $ as the solution to %
\eqref{e:static-stokes} with the estimate \eqref{stokes}, the respective
structure and pressure variables $\omega _{1},\omega _{2}$ and $\pi _{0}$
given by \eqref{e:resolvent_r1},\eqref{assign} along with the estimate %
\eqref{plate} (and where $\widehat{\omega }_{1}$ is defined by \eqref{p_1});
(\ref{d}) characterizes the pressure $\pi _{0}$ in terms of the variables $%
\omega _{1}$ and $\mu $. This shows that the solution of %
\eqref{e:resolvent-system} actually belongs to ${\mathcal{D}}({\mathcal{A}}%
_{\rho })$. In short, $0\in \rho (\mathcal{A}_{\rho })$, which concludes the
proof. 
\end{proof}

% Subsection

\subsection{$\protect\lambda=i\protect\beta$ is in the resolvent set $%
\protect\rho({\mathcal{A}})$}

Let us recall the expression of the dynamics operator semigroup $\mathcal{A}%
_{\rho }$ in \eqref{domain}. In straightforward fashion, one can then
compute the % Hilbert space 
associated adjoint operator $\mathcal{A}_{\rho }^{\ast }:D(\mathcal{A}_{\rho
})\subset \mathbf{H}_{\rho }\rightarrow \mathbf{H}_{\rho }$ to be 
\begin{equation}
\mathcal{A}_{\rho }^{\ast }\equiv 
\begin{bmatrix}
0 & -I & 0 \\ 
P_{\rho }^{-1}\Delta ^{2}-P_{\rho }^{-1}G_{\rho ,1}|_{\Omega } & 0 & P_{\rho
}^{-1}G_{\rho ,2}|_{\Omega } \\ 
\nabla G_{\rho ,1} & 0 & \Delta -\nabla G_{\rho ,2}%
\end{bmatrix}%
\,,  \label{A_*}
\end{equation}%
with $D(\mathcal{A}_{\rho }^{\ast })=D(\mathcal{A}_{\rho })$. %\label{D*}
The above operator will be utilized in the proof of the following result.

\begin{proposition}
Let $\sigma ({\mathcal{A}}_{\rho })$ be the spectrum of the dynamics
operator ${\mathcal{A}}_{\rho }$ defined by \eqref{domain}-%
\eqref{e:domain-of-generator}. Then $i\mathbb{R}\cap \sigma ({\mathcal{A}}%
_{\rho })=\emptyset $.
\end{proposition}

\begin{proof}
Let $\sigma _{p}({\mathcal{A}}_{\rho })$, $\sigma _{r}({\mathcal{A}}_{\rho })
$, $\sigma _{r}({\mathcal{A}}_{\rho })$ denote---respectively---the point,
continuous and residual spectrum of the operator ${\mathcal{A}}$. \newline
1. (\emph{Point spectrum}) We aim at showing that $i\mathbb{R}\cap \sigma
_{p}(\mathcal{A}_{\rho })=\emptyset $. Given $\beta \in \mathbb{R}\setminus
\{0\}$, we consider the equation 
\begin{equation}
\mathcal{A}_{\rho }\left[ 
\begin{array}{c}
\omega _{1} \\ 
\omega _{2} \\ 
\mu 
\end{array}%
\right] =i\beta \left[ 
\begin{array}{c}
\omega _{1} \\ 
\omega _{2} \\ 
\mu 
\end{array}%
\right]   \label{relate}
\end{equation}%
for some $[\omega _{1},\omega _{2},\mu ]\in D(\mathcal{A}_{\rho })$.
Moreover, we set%
\begin{equation}
\pi _{0}\equiv G_{\rho ,1}(w_{1})+G_{\rho ,2}(u).  \label{pi}
\end{equation}%
Taking the inner product of both sides, and subsequently integrating by
parts, then it follows 
\begin{eqnarray}
i\beta \left\Vert \left[ 
\begin{array}{c}
\omega _{1} \\ 
\omega _{2} \\ 
\mu 
\end{array}%
\right] \right\Vert _{\mathbf{H}_{\rho }}^{2} &=&\left( \mathcal{A}_{\rho }%
\left[ 
\begin{array}{c}
\omega _{1} \\ 
\omega _{2} \\ 
\mu 
\end{array}%
\right] ,%
\begin{array}{c}
\omega _{1} \\ 
\omega _{2} \\ 
\mu 
\end{array}%
\right)   \notag \\
&&  \notag \\
&=&(\Delta \omega _{2},\Delta \omega _{1})_{\Omega }+(-\Delta ^{2}\omega
_{1}+\left. \pi _{0}\right\vert _{\Omega },\omega _{2})_{\Omega }+\left(
\Delta \mu -\nabla \pi _{0},\mu \right) _{\mathcal{O}}  \notag \\
&&  \notag \\
&=&(\Delta \omega _{2},\Delta \omega _{1})_{\Omega }+\left( \nabla \Delta
\omega _{1},\nabla \omega _{2}\right) _{\Omega }+(\left. \pi _{0}\right\vert
_{\Omega }(0,0,1),(\mu ^{1},\mu ^{2},\omega _{2}))_{\Omega }  \notag \\
&&\text{ \ }-(\nabla \mu ,\nabla \mu )_{\mathcal{O}}+\left\langle \frac{%
\partial \mu }{\partial \nu },\mu \right\rangle _{\Omega }-\left\langle \pi
_{0}\nu ,\mu \right\rangle _{\Omega }  \notag \\
&&  \notag \\
&=&(\Delta \omega _{2},\Delta \omega _{1})_{\Omega }-\left( \Delta \omega
_{1},\Delta \omega _{2}\right) _{\Omega }-(\nabla \mu ,\nabla \mu )_{%
\mathcal{O}}  \notag \\
&&+\left( \left[ 
\begin{array}{c}
\partial _{x_{3}}\mu ^{1} \\ 
\partial _{x_{3}}\mu ^{2} \\ 
\partial _{x_{3}}\mu ^{3}%
\end{array}%
\right] ,\left[ 
\begin{array}{c}
0 \\ 
0 \\ 
\mu ^{3}%
\end{array}%
\right] \right) _{\Omega },  \label{long}
\end{eqnarray}%
or%
\begin{equation}
i\beta \left\Vert \left[ 
\begin{array}{c}
\omega _{1} \\ 
\omega _{2} \\ 
\mu 
\end{array}%
\right] \right\Vert _{\mathbf{H}_{\rho }}^{2}=-\left\Vert \nabla \mu
\right\Vert _{\mathcal{O}}^{2}-2i\,\text{Im}(\Delta \omega _{1},\Delta
\omega _{2})_{\Omega };  \label{long_2}
\end{equation}

whence we obtain 
\begin{equation}
\mu =0\quad \text{ in }\mathcal{O}\,.  \label{fluid}
\end{equation}

In turn, the boundary condition $\mu =(0,0,\omega _{2})$ on $\Omega $,
intrinsic to elements of ${\mathcal{D}}(\mathcal{A}_{\rho })$, yields as
well 
\begin{equation}
\omega _{2}=0\quad \text{ in }\Omega \,.  \label{velocity}
\end{equation}

And further in turn, the first component relation in (\ref{relate}),
combined with the appearance of $\mathcal{A}_{\rho }$ in (\ref{domain}),
yield $i\beta \omega_1=\omega_2$. Hence for $\beta \ne 0$, 
\begin{equation}  \label{displacement}
\omega _{1}=0 \quad \text{ in }\Omega\,.
\end{equation}
The relations (\ref{fluid}), (\ref{velocity}) and (\ref{displacement}) give
the conclusion that $i\beta $ is not an eigenvalue of $\mathcal{A}_{\rho }$.

\medskip \noindent 2. (\emph{Residual spectrum}) We aim at showing that $i%
\mathbb{R}\cap\sigma_r(\mathcal{A}_{\rho})=\emptyset$. Given $\beta \in 
\mathbb{R}\setminus \{0\}$, if $i\beta $ is in the residual spectrum of $%
\mathcal{A}_{\rho }$, then necessarily $i\beta $ is in the point spectrum of 
$\mathcal{A}_{\rho }^{\ast }:D(\mathcal{A}_{\rho}^{\ast })\subset \mathbf{H}%
_{\rho }\rightarrow\mathbf{H}_{\rho }$; see e.g., \cite[p.~127]{friedman}.
In this case, given the appearance and the domain of $\mathcal{A}%
_{\rho}^{\ast }$ in \eqref{A_*}, we proceed \emph{verbatim} along the lines
of Step~1. to deduce that $i\mathbb{R}\cap \sigma_r(\mathcal{A}%
_{\rho})=\emptyset$.

\medskip \noindent 3. (\emph{Continuous spectrum}) This is by far the most
challenging part of the proof. To make the inference that $i\mathbb{R}$ has
empty intersection with the continuous spectrum, it is enough to show that $i%
\mathbb{R}$ does not intersect with the \emph{approximate spectrum}; see
e.g., \cite[p.~128]{friedman}.

To this end, let $\beta \in \mathbb{R}\setminus \{0\}$ be given. If $i\beta $
is in the approximate spectrum of $\mathcal{A}_{\rho }$, then by definition
there exists a sequence 
\begin{eqnarray}
\left\{ \left[ 
\begin{array}{c}
\omega_{1,n} \\ 
\omega_{2,n} \\ 
\mu_{n}%
\end{array}
\right] \right\}_{n=1}^{\infty} &\subset & {\mathcal{D}}(\mathcal{A}_{\rho
})\; \text{such that for all $n$:} \quad \left\Vert \left[ 
\begin{array}{c}
\omega _{1,n} \\ 
\omega _{2,n} \\ 
\mu _{n}%
\end{array}
\right] \right\Vert _{\mathbf{H}_{\rho }}=1  \notag \\
&&  \notag \\
\text{and }\left[ 
\begin{array}{c}
\omega _{1,n}^{\ast } \\ 
\omega _{2,n}^{\ast } \\ 
\mu _{n}^{\ast }%
\end{array}
\right] &=&\left( i\beta -\mathcal{A}_{\rho }\right) \left[ 
\begin{array}{c}
\omega _{1,n} \\ 
\omega _{2,n} \\ 
\mu _{n}%
\end{array}
\right] \text{ satisfies }\left\Vert \left[ 
\begin{array}{c}
\omega _{1,n}^{\ast } \\ 
\omega _{2,n}^{\ast } \\ 
\mu _{n}^{\ast }%
\end{array}
\right] \right\Vert _{\mathbf{H}_{\rho }}<\frac{1}{n}\,.  \label{approx}
\end{eqnarray}
We consider therewith the relation 
\begin{equation}  \label{r2-approximate}
\left(i\beta -\mathcal{A}_{\rho }\right) \left[ 
\begin{array}{c}
\omega _{1,n} \\ 
\omega _{2,n} \\ 
\mu _{n}%
\end{array}
\right] =\left[ 
\begin{array}{c}
\omega _{1,n}^{\ast } \\ 
\omega _{2,n}^{\ast } \\ 
\mu _{n}^{\ast }%
\end{array}
\right]\,.
\end{equation}

In PDE terms, each $[\omega _{1,n},\omega _{2,n},\mu _{n}]$ satisfies the
following problem: 
\begin{subequations}
\label{e:approximate-system}
\begin{align}
& i\beta \omega _{1,n}-\omega _{2,n}=\omega _{1,n}^{\ast } & & \text{in $%
\Omega $}  \label{s1} \\
& i\beta \omega _{2,n}+P_{\rho }^{-1}\Delta ^{2}\omega _{1,n}-P_{\rho
}^{-1}p_{n}\big|_{\Omega }=\omega _{2,n}^{\ast } & & \text{ in $\Omega $}
\label{s2} \\
& \omega _{1,n}\big|_{\Omega }=\frac{\partial \omega _{1,n}}{\partial \nu }%
\Big|_{\Omega }=0 & & \text{ on $\partial \Omega $}  \label{s3} \\
& i\beta \mu _{n}-\Delta \mu _{n}+\nabla p_{n}=\mu _{n}^{\ast } & & \text{in 
$\mathcal{O}$}  \label{s4} \\
& \mathrm{div}(\mu _{n})=0 & & \text{in $\mathcal{O}$}  \label{s5} \\
& \mu _{n}=0 & & \text{on $S$}  \label{s6} \\
& \mu _{n}=(0,0,\omega _{2,n}) & & \text{on $\Omega $,}  \label{s7}
\end{align}%
where for each $n$, the associated pressure term is given by 
\end{subequations}
\begin{equation}
p_{n}=G_{1}\omega _{1,n}+G_{2}\mu _{n}\,.  \label{p_seq}
\end{equation}%
Multiplying both parts of the expression (\ref{r2-approximate}) by $[\omega
_{1,n},\omega _{2,n},\mu _{n}]$ and integrating by parts gives 
\begin{equation*}
\left\Vert \nabla \mu _{n}\right\Vert _{\mathcal{O}}^{2}=\left( \left[ 
\begin{array}{c}
\omega _{1,n}^{\ast } \\ 
\omega _{2,n}^{\ast } \\ 
\mu _{n}^{\ast }%
\end{array}%
\right] ,\left[ 
\begin{array}{c}
\omega _{1,n} \\ 
\omega _{2,n} \\ 
\mu _{n}%
\end{array}%
\right] \right) _{\mathbf{H}_{\rho }}-i\beta \left\Vert \left[ 
\begin{array}{c}
\omega _{1,n} \\ 
\omega _{2,n} \\ 
\mu _{n}%
\end{array}%
\right] \right\Vert _{\mathbf{H}_{\rho }}^{2}-2i\text{Im}(\Delta \omega
_{1,n},\Delta \omega _{2,n})_{\Omega }\,.
\end{equation*}%
We have then from (\ref{approx}) that 
\begin{equation}
\mu _{n}\longrightarrow 0\quad \text{(strongly) in $\mathbf{H}^{1}(\mathcal{O%
})$}.  \label{f_con}
\end{equation}

In turn, from the boundary condition \eqref{s7} and the Sobolev Embedding
Theorem, we have 
\begin{equation*}
\Vert \omega _{2,n}\Vert _{H^{1/2}(\Omega )}=\Vert \mu _{n}^{3}\Vert
_{H^{1/2}(\Omega )}\leq C\,\Vert \mu _{n}\Vert _{\mathbf{H}^{1}(\mathcal{O}%
)}\,,
\end{equation*}%
whence 
\begin{equation}
\omega _{2,n}\longrightarrow 0\quad \text{(strongly) in $H^{1/2}(\Omega )$}%
\,.  \label{v_con}
\end{equation}

At this point, we invoke the unique decomposition 
\begin{equation}
p_{n}=c_{n}+q_{n}\,,  \label{decomp}
\end{equation}%
where for each $n$, 
\begin{equation}
c_{n}=\text{constant}\,;\quad q_{n}\in L^{2}(\mathcal{O})/\mathbb{R}\,.
\label{part}
\end{equation}%
Then, from the known regularity for Stokes flow---see, e.g., estimate (2.46)
in \cite[p.~23]{temam}---we have from (\ref{s4})-(\ref{s6}) 
\begin{eqnarray}
\Vert q_{n}\Vert _{L^{2}(\mathcal{O})/\mathbb{R}} &\leq &C\left( \Vert
i\beta \mu _{n}\Vert _{\mathbf{L}^{2}(\mathcal{O)}}+\Vert \mu _{n}\Vert _{%
\mathbf{H}^{1/2}(\partial \mathcal{O)}}+\Vert \mu _{n}^{\ast }\Vert _{%
\mathbf{L}^{2}(\mathcal{O)}}\right)  \notag \\
&\leq &C_{\beta }\left( \Vert \mu _{n}\Vert _{\mathbf{H}^{1}(\mathcal{O)}%
}+\Vert \mu _{n}^{\ast }\Vert _{\mathbf{L}^{2}(\mathcal{O)}}\right) \,;
\label{q0}
\end{eqnarray}%
whence we obtain from (\ref{f_con}) and (\ref{approx}), 
\begin{equation}
q_{n}\longrightarrow 0\quad \text{ strongly in}\;L^{2}(\mathcal{O})\,.
\label{q1}
\end{equation}%
Moreover, since each $q_{n}$ is harmonic \textit{a fortiori}, we have
available the boundary trace estimate 
\begin{eqnarray}
\Vert q_{n}|_{\partial \mathcal{O}}\Vert _{H^{-1/2}(\partial \mathcal{O})}
&\leq &C\,\Vert q_{n}\Vert _{L^{2}(\mathcal{O)}}  \notag \\
&\leq &C_{\beta }\,\left( \Vert \mu _{n}\Vert _{\mathbf{H}^{1}(\mathcal{O)}%
}+\Vert \mu _{n}^{\ast }\Vert _{\mathcal{H}_{\mathrm{f}}}\right)  \label{q2}
\end{eqnarray}%
(see e.g., \cite[Proposition~1]{dvorak}; in attaining the second estimate we
have also used (\ref{q0})); appealing again to (\ref{f_con}) and (\ref%
{approx}) we then have 
\begin{equation}
q_{n}|_{\partial \mathcal{O}}\longrightarrow 0\quad \text{strongly in}%
\;H^{-1/2}(\mathcal{O})\,.  \label{q3}
\end{equation}

\medskip Now using the decomposition (\ref{decomp}) in the structural
equation (\ref{s2}), we have for all $n$, 
\begin{equation*}
c_{n}=-q_{n}|_{\Omega }+\Delta ^{2}\omega _{1,n}+i\beta P_{\rho }\omega
_{2,n}-P_{\rho }\omega _{2,n}^{\ast }\,,
\end{equation*}%
and so a measurement in the $H^{-2}(\Omega )$-topology gives%
\begin{equation}
\begin{split}
& c_{n}\Vert 1\Vert _{H^{-2}(\Omega )}=\big\|-q_{n}|_{\Omega }+\Delta
^{2}\omega _{1,n}+i\beta P_{\rho }\omega _{2,n}-P_{\rho }\omega _{2,n}^{\ast
}\big\|_{H^{-2}(\Omega )} \\
& \qquad \qquad \qquad \leq C_{\beta }\,\Big(\Vert q_{n}|_{\Omega }\Vert
_{H^{-1/2}(\Omega )}+\Vert \omega _{1,n}\Vert _{H_{0}^{2}(\Omega )}+\Vert
\omega _{2,n}\Vert _{L^{2}(\Omega )}+\big\|\omega _{2,n}^{\ast }\big\|%
_{D(P_{\rho }^{1/2})}\Big)\,.
\end{split}
\label{c1}
\end{equation}

Combining (\ref{approx}), (\ref{q3}) and (\ref{v_con}) with (\ref{c1}) we
achieve the conclusion that 
\begin{equation*}
\{c_{n}\}_{n\geq 1}\quad \text{is uniformly bounded in $n$.}
\end{equation*}%
Hence, there is a subsequence of constants---still denoted as $%
\{c_{n}\}_{n\geq 1}$---which satisfies for some $\tilde{c}$ 
\begin{equation}
c_{n}\rightarrow \tilde{c}\quad \text{(strongly) in $\mathbb{C}$.}
\label{c_con}
\end{equation}

We now turn our attention to the mechanical system (\ref{s2})-(\ref{s3}),
that is 
\begin{equation*}
\begin{cases}
\Delta ^{2}\omega _{1,n}=p_{n}|_{\Omega }-i\beta P_{\rho }\,\omega
_{2,n}+P_{\rho }\,\omega _{2,n}^{\ast } & \text{in $\Omega $} \\[1mm] 
\omega _{1,n}=\dfrac{\partial \omega _{1,n}}{\partial \nu }=0 & \text{on $%
\partial \Omega $.}%
\end{cases}%
\end{equation*}%
%
%
% We intend to apply this convergence to (\ref{s2})-(\ref{s3}): we have
By way of looking at this sequence of boundary value problems, let us invoke
the realization $A$ of the bilaplacian operator, defined by $A\varphi
:=\Delta ^{2}\varphi $\thinspace , $\varphi \in {\mathcal{D}}%
(A)=H^{4}(\Omega )\cap H_{0}^{2}(\Omega )$. Then we have abstractly 
\begin{equation*}
A\omega _{1,n}=c_{n}+q_{n}|_{\Omega }-i\beta P_{\rho }\,\omega
_{2,n}+P_{\rho }\,\omega _{2,n}^{\ast }\in \big[{\mathcal{D}}(A^{1/2})\big]%
^{\prime }\,,
\end{equation*}%
where ${\mathcal{D}}(A^{1/2})=H_{0}^{2}(\Omega )$.

Applying the inverse $A^{-1}\in {\mathcal{L}}(L^{2}(\Omega ),{\mathcal{D}}%
(A))$ to both sides of the above equality gives 
\begin{equation}
\omega _{1,n}=A^{-1}c_{n}+A^{-1}\big(q_{n}|_{\Omega }-i\beta P_{\rho
}\,\omega _{2,n}+P_{\rho }\omega _{2,n}^{\ast }\big)\in {\mathcal{D}}%
(A^{1/2})\,.  \label{m1}
\end{equation}%
Subsequently we can then pass to the limit in (\ref{m1}) (meanwhile using %
\eqref{c_con}, \eqref{q3}, \eqref{v_con} and \eqref{approx}) so as to have 
\begin{equation}
\tilde{\omega}=\lim_{n\rightarrow \infty }\omega _{1,n}=\lim_{n\rightarrow
\infty }A^{-1}\,c_{n}=A^{-1}\tilde{c}\,.  \label{m1.5}
\end{equation}%
Thus, this (structural) limit must satisfy 
\begin{equation}
\Delta ^{2}\tilde{\omega}=\tilde{c}\quad \text{in $\Omega $},\qquad \tilde{%
\omega}=\frac{\partial \tilde{\omega}}{\partial \nu }=0\quad \text{on $%
\partial \Omega $}.  \label{m2}
\end{equation}%
Now since $\omega _{1,n}\in H_{0}^{2}(\Omega )\cap \frac{L^{2}(\Omega )}{%
\mathbb{R}}$ for every $n$, then so is strong limit $\tilde{\omega}$. But
from (\ref{m2}) and the characterization (\ref{constant}), we have also that 
$\tilde{\omega}\in \left[ H_{0}^{2}(\Omega )\cap \frac{L^{2}(\Omega )}{%
\mathbb{R}}\right] ^{\bot }$. Thus,%
\begin{equation}
\lim_{n\rightarrow \infty }\omega _{1,n}=0.  \label{m3}
\end{equation}%
Finally, from (\ref{s1}), 
\begin{equation*}
\omega _{2,n}=i\beta \omega _{1,n}-\omega _{1,n}^{\ast };
\end{equation*}%
whence we obtain with\eqref{approx} and (\ref{m3}), 
\begin{equation}
\lim_{n\rightarrow \infty }\omega _{2,n}=0\text{ in ${\mathcal{D}}(P_{\rho
}^{1/2})$.}  \label{mech_z}
\end{equation}%
The limits in (\ref{m3}) and \eqref{mech_z}, combined with the one in %
\eqref{f_con}, now contradict the fact from \eqref{approx} that 
\begin{equation*}
\big\|\lbrack \omega _{1,n},\omega _{2,n},\mu _{n}]\big\|_{\mathbf{H}_{\rho
}}=1\qquad \forall n\,.
\end{equation*}%
Since $\beta \in \mathbb{R}\setminus \{0\}$ was arbitrary, we conclude that
the approximate spectrum of ${\mathcal{A}}_{\rho }$ does not intersect with $%
i\mathbb{R}$. 
\end{proof}

% Section -- Main result

\section{Proof of Theorem~\protect\ref{t:main} (Main result)}

%\begin{proof}
Here we will utilize Theorem~\ref{t:borichev-tomilov} (see \cite[Theorem~2.4]%
{borichev}) in the case currently being considered; namely, $\rho >0$, so
that rotational forces are accounted for in the fluid-structure PDE
dynamics. By way of using the aforesaid resolvent criterion, we consider
arbitrary data $[\omega _{1}^{\ast },\omega _{2}^{\ast },u^{\ast }]\in 
\mathbf{H}_{\rho }$, and the corresponding pre-image $[\omega _{1},\omega
_{2},u]\in {\mathcal{D}}({\mathcal{A}}_{\rho })$ which solves the following
relation for given $\beta \in \mathbb{R}$: 
\begin{equation}
(i\beta -\mathcal{A}_{\rho })\left[ 
\begin{array}{c}
\omega _{1} \\ 
\omega _{2} \\ 
\mu%
\end{array}%
\right] =\left[ 
\begin{array}{c}
\omega _{1}^{\ast } \\ 
\omega _{2}^{\ast } \\ 
\mu ^{\ast }%
\end{array}%
\right] \in \mathbf{H}_{\rho }\,.  \label{resolvent}
\end{equation}%
With respect to this relation, the proof of Theorem~\ref{t:main} will be
established if we derive the following estimate for $|\beta |$ sufficiently
large (and a positive constant $C$): 
\begin{equation}
\left\Vert \left[ 
\begin{array}{c}
\omega _{1} \\ 
\omega _{2} \\ 
\mu%
\end{array}%
\right] \right\Vert _{\mathbf{H}_{\rho }}\leq C\,|\beta |\,\left\Vert \left[ 
\begin{array}{c}
\omega _{1}^{\ast } \\ 
\omega _{2}^{\ast } \\ 
\mu ^{\ast }%
\end{array}%
\right] \right\Vert _{\mathbf{H}_{\rho }}\,;  \label{object}
\end{equation}%
this is the frequency estimate \eqref{e:asymptotic-tomilov} with $\alpha =1$.

Using the definition of ${\mathcal{A}}_{\rho}:{\mathcal{D}}({\mathcal{A}}%
_{\rho })\subset \mathbf{H}_{\rho }\rightarrow \mathbf{H}_{\rho }$, this
gives 
\begin{align*}
i\beta \omega_1-\omega_2 &=\omega_1^{\ast} \qquad\text{in $\Omega$} \\
i\beta \omega_2+P_{\rho
}^{-1}\,\Delta^2\omega_1-P_{\rho}^{-1}G_{\rho,1}\omega_1|_{\Omega }
-P_{\rho}^{-1}G_{\rho,2}\mu|_{\Omega} &=\omega_2^{\ast} \qquad \text{in $%
\Omega$} \\
i\beta \mu -\Delta \mu +\nabla G_{\rho,1}\omega_1+\nabla G_{\rho,2}\mu
&=\mu^{\ast}\qquad \text{in ${\mathcal{O}}$.}
\end{align*}
\newline
Upon a rearrangement and setting pressure variable 
\begin{equation}  \label{p-set}
\pi \equiv G_{\rho ,1}\omega _{1}+G_{\rho ,2}\mu\,,
\end{equation}
we then have 
\begin{align*}
\omega_2 &= i\beta \omega_1-\omega_1^{\ast } & \text{in $\Omega$}  \notag \\
-\beta^2\omega_1-i\beta \omega_1^{\ast} +P_{\rho
}^{-1}\Delta^2\omega_1-P_{\rho }^{-1} \pi|_{\Omega } &=\omega_2^{\ast} & 
\text{in $\Omega$}  \notag \\
i\beta \mu -\Delta \mu +\nabla \pi & =\mu^{\ast} & \text{in ${\mathcal{O}}$.}
\end{align*}
We have then following (static) fluid-structure PDE system: 
\begin{subequations} \label{e:static-beta}
\begin{align}
\omega_2 &=i\beta \omega_1-\omega_1^{\ast} & \text{in $\Omega$}  \label{r3}
\\
-\beta^2\,P_{\rho }\omega_1+\Delta^2\omega_1-\pi|_{\Omega } &=
P_{\rho}\,\omega_2^{\ast}+i\beta P_{\rho}\,\omega_1^{\ast } & \text{in $%
\Omega$}  \label{r4} \\
\omega_1|_{\partial \Omega}& =\frac{\partial\omega_1}{\partial n}\Big|%
_{\partial \Omega }=0 & \text{on $\partial\Omega$}  \label{r5} \\
i\beta \mu -\Delta \mu +\nabla \pi &=\mu ^{\ast} & \text{in ${\mathcal{O}}$}
\label{r6} \\
\mathrm{div}(\mu ) &=0 & \text{in ${\mathcal{O}}$}  \label{r7} \\
\mu &=0\  & \text{on $S$}  \notag \\
\mu &=[\mu ^{1},\mu ^{2},\mu ^{3}]=[0,0,i\beta\omega_1-\omega_1^{\ast }] & 
\text{on $\Omega$.}  \label{r8}
\end{align}

\bigskip \noindent \emph{Step 1. (An estimate for the fluid gradient)} Let
us return to the resolvent equation \eqref{resolvent}. It is easily seen
that an integration by parts gives the following static dissipation
relation: 
\end{subequations}
\begin{eqnarray*}
&&\left( (i\beta -\mathcal{A}_{\rho })\left[ 
\begin{array}{c}
\omega _{1} \\ 
\omega _{2} \\ 
\mu%
\end{array}%
\right] ,\left[ 
\begin{array}{c}
\omega _{1} \\ 
\omega _{2} \\ 
\mu%
\end{array}%
\right] \right) _{\mathbf{H}_{\rho }} \\
&=&i\beta \left\Vert \left[ 
\begin{array}{c}
\omega _{1} \\ 
\omega _{2} \\ 
\mu%
\end{array}%
\right] \right\Vert _{\mathbf{H}_{\rho }}^{2}+2i\text{Im}(\Delta \omega
_{1},\Delta \omega _{2})_{\Omega }+\left\Vert \nabla \mu \right\Vert _{%
\mathcal{O}}^{2}\,;
\end{eqnarray*}%
see (\ref{long})-(\ref{long_2}). Combining this with relation (\ref%
{resolvent}), we then have 
\begin{equation*}
\begin{split}
& i\beta \,\left\Vert \left[ 
\begin{array}{c}
\omega _{1} \\ 
\omega _{2} \\ 
\mu%
\end{array}%
\right] \right\Vert _{\mathbf{H}_{\rho }}^{2}+2i\,\text{Im}(\Delta \omega
_{1},\Delta \omega _{2})_{\Omega }+\left\Vert \nabla \mu \right\Vert _{{%
\mathcal{O}}}^{2}= \\
& \qquad \qquad \qquad =\left( (i\beta -\mathcal{A}_{\rho })\left[ 
\begin{array}{c}
\omega _{1} \\ 
\omega _{2} \\ 
\mu%
\end{array}%
\right] ,\left[ 
\begin{array}{c}
\omega _{1} \\ 
\omega _{2} \\ 
\mu%
\end{array}%
\right] \right) _{\mathbf{H}_{\rho }}=\left( \left[ 
\begin{array}{c}
\omega _{1}^{\ast } \\ 
\omega _{2}^{\ast } \\ 
\mu ^{\ast }%
\end{array}%
\right] ,\left[ 
\begin{array}{c}
\omega _{1} \\ 
\omega _{2} \\ 
\mu%
\end{array}%
\right] \right) _{\mathbf{H}_{\rho }},
\end{split}%
\end{equation*}%
whence we obtain 
\begin{equation}
\left\Vert \nabla \mu \right\Vert _{L^{2}({\mathcal{O}})}^{2}=\text{Re}%
\,\left( \left[ 
\begin{array}{c}
\omega _{1}^{\ast } \\ 
\omega _{2}^{\ast } \\ 
\mu ^{\ast }%
\end{array}%
\right] ,\left[ 
\begin{array}{c}
\omega _{1} \\ 
\omega _{2} \\ 
\mu%
\end{array}%
\right] \right) _{\mathbf{H}_{\rho }}\,.  \label{dissi}
\end{equation}

\medskip \noindent \emph{Step 2.} \emph{(Control of the $\beta $-mechanical
displacement in a lower topology)} % This comes quickly: 
Using the fluid Dirichlet boundary condition in \eqref{r8} we have 
\begin{equation*}
i\beta \omega _{1}=\mu ^{3}\big|_{\Omega }+\omega _{1}^{\ast }\,.
\end{equation*}%
We estimate this expression by invoking in sequence, the Sobolev Embedding
Theorem, Poincar\'{e}'s inequality and \eqref{dissi}. In this way, we then
obtain 
\begin{eqnarray}
\Vert \beta \omega _{1}\Vert _{H^{1/2}(\Omega )} &\leq &\big\|\mu ^{3}\big|%
_{\Omega }+\omega _{1}^{\ast }\big\|_{H^{1/2}(\Omega )}  \notag \\[1mm]
&\leq &C\,\Big(\Vert \nabla \mu \Vert _{L^{2}({\mathcal{O}})}+\Vert \omega
_{1}^{\ast }\Vert _{H_{0}^{2}(\Omega )}\Big)  \notag \\[1mm]
&\leq &C\,\left( \sqrt{\left\vert \text{Re}\left( \left[ 
\begin{array}{c}
\omega _{1}^{\ast } \\ 
\omega _{2}^{\ast } \\ 
\mu ^{\ast }%
\end{array}%
\right] ,\left[ 
\begin{array}{c}
\omega _{1} \\ 
\omega _{2} \\ 
\mu%
\end{array}%
\right] \right) _{\mathbf{H}_{\rho }}\right\vert }+\left\Vert \omega
_{1}^{\ast }\right\Vert _{H_{0}^{2}(\Omega )}\right) \,.  \label{est2.5}
\end{eqnarray}

\medskip \noindent \emph{Step 3. (Control of the mechanical displacement)}
We multiply both sides of the mechanical equation in \eqref{r4} by $\omega_1$
and integrate. This gives the relation 
\begin{equation}  \label{mech}
\big(\Delta^2\omega_1,\omega_1\big)_{L^2(\Omega)}= \beta^2\big\|%
P_{\rho}^{1/2}\omega_1\big\|_{L^2(\Omega)}^2 +(\pi|_{\Omega
},\omega_1)_{\Omega} +(P_{\rho}\omega_2^*+i\beta P_{\rho
}\omega_1^*,\omega_1)_{L^2(\Omega)}\,.
\end{equation}

\medskip \noindent \emph{(3.i)} To handle the first term on the right hand
side of \eqref{mech}, we invoke Poincar\'{e}'s Inequality, thereby obtaining 
\begin{equation}
\beta ^{2}\big\|P_{\rho }^{1/2}\omega _{1}\big\|_{L^{2}(\Omega )}^{2}=\beta
^{2}\Big(\Vert \omega _{1}\Vert _{L^{2}(\Omega )}^{2}+\rho \big\|\nabla
\omega _{1}\big\|_{L^{2}(\Omega )}^{2}\Big)\leq C_{\rho }\beta ^{2}\big\|%
\nabla \omega _{1}\big\|_{L^{2}(\Omega )}^{2}\,.  \label{inter}
\end{equation}%
Now, 
\begin{equation*}
\begin{split}
\beta ^{2}\big\|\nabla \omega _{1}\big\|_{L^{2}(\Omega )}^{2}& =\beta \big(%
\nabla \omega _{1},\beta \nabla \omega _{1}\big)_{L^{2}(\Omega )}=\beta \big(%
\nabla \omega _{1},\beta \nabla \omega _{1}\big)_{H^{1/2}(\Omega
),H^{-1/2}(\Omega )} \\[1mm]
& \leq C\,|\beta |\,\Vert \omega _{1}\Vert _{H^{3/2}(\Omega )}\,\Vert \beta
\omega _{1}\Vert _{H^{1/2}(\Omega )}\,.
\end{split}%
\end{equation*}%
Subsequently, interpolating between $H^{2}(\Omega )$ and $H^{1/2}(\Omega )$
with interpolation parameter $\theta =\frac{1}{3}$ (see e.g., \cite{L-M}),
we obtain 
\begin{equation*}
\begin{split}
\beta ^{2}\big\|\nabla \omega _{1}\big\|_{L^{2}(\Omega )}^{2}& \leq
C\,|\beta |^{2/3}\,\Vert \beta |^{1/3}|\omega _{1}\Vert _{H^{3/2}(\Omega
)}\Vert \beta \omega _{1}\Vert _{H^{1/2}(\Omega )} \\[1mm]
& \leq C\,|\beta |^{2/3}\Big[\Vert \omega _{1}\Vert _{H^{2}(\Omega
)}^{2/3}\,\Vert \beta \omega _{1}\Vert _{H^{1/2}(\Omega )}^{1/3}\Big]\,\Vert
\beta \omega _{1}\Vert _{H^{1/2}(\Omega )}\,.
\end{split}%
\end{equation*}%
Via Young's inequality, with conjugate exponents $3$ and $3/2$, we then have 
\begin{equation*}
\beta ^{2}\big\|\nabla \omega _{1}\big\|_{L^{2}(\Omega )}^{2}\leq C\,\Vert
\omega _{1}\Vert _{H^{2}(\Omega )}^{2/3}\,|\beta |^{2/3}\,\Vert \beta \omega
_{1}\Vert _{H^{1/2}(\Omega )}^{4/3}\leq \epsilon \Vert \omega _{1}\Vert
_{H^{2}(\Omega )}^{2}+C_{\epsilon }|\beta |\,\Vert \beta \omega _{1}\Vert
_{H^{1/2}(\Omega )}^{2}\,;
\end{equation*}%
subsequently reinvoking the estimate \eqref{est2.5}, we so have for $|\beta
|>1$, 
\begin{eqnarray}
\beta ^{2}\big\|\nabla \omega _{1}\big\|_{L^{2}(\Omega )}^{2} &\leq
&\epsilon \Vert \omega _{1}\Vert _{H^{2}(\Omega )}^{2}+C_{\epsilon }\,|\beta
|\left( \left\Vert \text{Re}\left( \left[ 
\begin{array}{c}
\omega _{1}^{\ast } \\ 
\omega _{2}^{\ast } \\ 
\mu ^{\ast }%
\end{array}%
\right] ,\left[ 
\begin{array}{c}
\omega _{1} \\ 
\omega _{2} \\ 
\mu%
\end{array}%
\right] \right) _{\mathbf{H}_{\rho }}\right\Vert +\Vert \omega _{1}^{\ast
}\Vert _{H_{0}^{2}(\Omega )}^{2}\right)  \notag \\
&\leq &2\,\epsilon \left\Vert \left[ 
\begin{array}{c}
\omega _{1} \\ 
\omega _{2} \\ 
\mu%
\end{array}%
\right] \right\Vert _{\mathbf{H}_{\rho }}^{2}+C_{\epsilon }|\beta
|^{2}\left\Vert \left[ 
\begin{array}{c}
\omega _{1}^{\ast } \\ 
\omega _{2}^{\ast } \\ 
\mu ^{\ast }%
\end{array}%
\right] \right\Vert _{\mathbf{H}_{\rho }}^{2}\,.  \label{est4}
\end{eqnarray}%
Applying the obtained estimate \eqref{est4} to the right hand side of %
\eqref{inter} yields now%
\begin{equation}
\beta ^{2}\big\|P_{\rho }^{1/2}\omega _{1}\big\|_{L^{2}(\Omega )}^{2}\leq
C_{\rho }\epsilon \,\left\Vert \left[ 
\begin{array}{c}
\omega _{1} \\ 
\omega _{2} \\ 
\mu%
\end{array}%
\right] \right\Vert _{\mathbf{H}_{\rho }}^{2}+C_{\epsilon }\,|\beta
|^{2}\left\Vert \left[ 
\begin{array}{c}
\omega _{1}^{\ast } \\ 
\omega _{2}^{\ast } \\ 
\mu ^{\ast }%
\end{array}%
\right] \right\Vert _{\mathbf{H}_{\rho }}^{2}\,.  \label{est4.5}
\end{equation}

\bigskip \noindent \emph{(3.ii)} To handle the second term on the right hand
side of \eqref{mech}, we observe that since $[\omega_1,\omega_2,\mu]\in 
\mathbf{H}_{\rho}$, then in particular 
\begin{equation*}
\int_{\Omega }\omega_1\, dx = 0\,.
\end{equation*}
In consequence, one has wellposedness of the following boundary value
problem (see \cite[Proposition~2.2]{temam}): 
\begin{equation}  \label{BVP}
\begin{cases}
-\Delta \psi +\nabla q=0 & \text{in ${\mathcal{O}}$} \\[1mm] 
\text{div}\,(\psi )=0 & \text{in ${\mathcal{O}}$} \\[1mm] 
\psi|_S=0 & \text{on $S$} \\[1mm] 
\psi|_\Omega =\big(\psi^1,\psi^2,\psi^3\big)\big|_{\Omega }=(0,0,\omega_1) & 
\text{on $\Omega$}\,,%
\end{cases}%
\end{equation}
with the estimate 
\begin{equation}  \label{bvp_est}
\big\|\nabla \psi \big\|_{\mathbf{L}^2({\mathcal{O}})}+\|q\|_{L^2({\mathcal{O%
}})}\le C\,\|\omega_1\|_{H^{1/2}(\Omega)}
\end{equation}
(implicitly, we are also using Poincar\'{e} inequality).

\medskip With this solution variable $\psi $ of \eqref{BVP} in hand, we now
address the second term on the right hand side of \eqref{mech}. Since the
normal vector $\nu $ equals $(0,0,1)$ on $\Omega $ (and as the fluid
variable $\mu $ is divergence free), we have

\begin{equation}
\begin{split}
(\pi |_{\Omega },\omega _{1})_{\Omega }& =-\left( \frac{\partial \mu }{%
\partial \nu },\left[ 
\begin{array}{c}
0 \\ 
0 \\ 
\omega _{1}%
\end{array}%
\right] \right) _{\mathbf{L}^{2}(\Omega )}+\left( \left. \pi \right\vert
_{\Omega }\nu ,\left[ 
\begin{array}{c}
0 \\ 
0 \\ 
\omega _{1}%
\end{array}%
\right] \right) _{\mathbf{L}^{2}(\Omega )} \\[1mm]
& =-\left( \frac{\partial \mu }{\partial \nu },\psi \right) _{\mathbf{L}%
^{2}(\partial {\mathcal{O}})}+\big(\left. \pi \right\vert _{\Omega }\nu
,\psi \big)_{\mathbf{L}^{2}(\partial {\mathcal{O}})}\,,
\end{split}
\label{p1}
\end{equation}%
after invoking the boundary conditions in \eqref{BVP}.

The use of Green's Identities and the Stokes system in \eqref{BVP} then
gives 
\begin{equation*}
\begin{split}
(\pi |_{\Omega },\omega _{1})_{\Omega }& =-\Big(\frac{\partial \mu }{%
\partial \nu },\psi \Big)_{\mathbf{L}^{2}(\partial {\mathcal{O}})}+\big(%
\left. \pi \right\vert _{\Omega }\nu ,\psi \big)_{\mathbf{L}^{2}(\partial {%
\mathcal{O}})} \\[1mm]
& =-\big(\Delta \mu ,\psi \big)_{\mathbf{L}^{2}({\mathcal{O}})}-\big(\nabla
\mu ,\nabla \psi \big)_{\mathbf{L}^{2}({\mathcal{O}})}+\big(\nabla \pi ,\psi %
\big)_{\mathbf{L}^{2}({\mathcal{O}})} \\[1mm]
& =-i\beta \,(u,\psi )_{\mathbf{L}^{2}({\mathcal{O}})}-\big(\nabla u,\nabla
\psi \big)_{\mathbf{L}^{2}({\mathcal{O}})}+\big(u^{\ast },\psi \big)_{%
\mathbf{L}^{2}({\mathcal{O}})}\,.
\end{split}%
\end{equation*}%
Estimating this right hand side by means of Poincar\'{e} Inequality, we then
have for $|\beta |>1$, 
\begin{equation}
\big|(\pi |_{\Omega },w_{1})_{\Omega }\big|\leq C\,|\beta |\,\Vert \nabla
\psi \Vert _{\mathbf{L}^{2}({\mathcal{O}})}\big(\Vert \nabla u\Vert _{%
\mathbf{L}^{2}({\mathcal{O}})}+\Vert u^{\ast }\Vert _{\mathbf{L}^{2}({%
\mathcal{O}})}\big)\,;  \label{k1}
\end{equation}%
and subsequently refining this inequality by means of \eqref{dissi}, %
\eqref{bvp_est} and \eqref{est2.5}, we establish 
\begin{equation}
\begin{split}
\big|(\pi |_{\Omega },w_{1})_{\Omega }\big|& \leq C\,|\beta |\,\Vert \omega
_{1}\Vert _{H^{1/2}(\Omega )}\,\left( \left\vert \text{Re}\left( \left[ 
\begin{array}{c}
\omega _{1}^{\ast } \\ 
\omega _{2}^{\ast } \\ 
\mu ^{\ast }%
\end{array}%
\right] ,\left[ 
\begin{array}{c}
\omega _{1} \\ 
\omega _{2} \\ 
\mu%
\end{array}%
\right] \right) _{\mathbf{H}_{\rho }}\right\vert ^{1/2}+\Vert u^{\ast }\Vert
_{\mathbf{L}^{2}({\mathcal{O}})}\right) \\[1mm]
& \leq C\,\left( \left\vert \text{Re}\left( \left[ 
\begin{array}{c}
\omega _{1}^{\ast } \\ 
\omega _{2}^{\ast } \\ 
\mu ^{\ast }%
\end{array}%
\right] ,\left[ 
\begin{array}{c}
\omega _{1} \\ 
\omega _{2} \\ 
\mu%
\end{array}%
\right] \right) _{\mathbf{H}_{\rho }}\right\vert ^{1/2}+\Vert \omega
_{1}^{\ast }\Vert _{H_{0}^{2}(\Omega )}\right) \cdot \\[1mm]
& \qquad \qquad \qquad \cdot \,C\,\left( \left\vert \text{Re}\left( \left[ 
\begin{array}{c}
\omega _{1}^{\ast } \\ 
\omega _{2}^{\ast } \\ 
\mu ^{\ast }%
\end{array}%
\right] ,\left[ 
\begin{array}{c}
\omega _{1} \\ 
\omega _{2} \\ 
\mu%
\end{array}%
\right] \right) _{\mathbf{H}_{\rho }}\right\vert ^{1/2}+\Vert u^{\ast }\Vert
_{\mathbf{L}^{2}({\mathcal{O}})}\right) \\[1mm]
& \leq \epsilon \left\Vert \left[ 
\begin{array}{c}
\omega _{1} \\ 
\omega _{2} \\ 
\mu%
\end{array}%
\right] \right\Vert _{\mathbf{H}_{\rho }}^{2}+C_{\epsilon }\,\left\Vert %
\left[ 
\begin{array}{c}
\omega _{1}^{\ast } \\ 
\omega _{2}^{\ast } \\ 
\mu ^{\ast }%
\end{array}%
\right] \right\Vert _{\mathbf{H}_{\rho }}^{2}\,,
\end{split}
\label{est5}
\end{equation}%
after again using Young's Inequality.

\medskip \noindent \emph{(3.iii)} It remains to handle the third term on the
right hand side of \eqref{mech}. By way of estimate \eqref{est2.5} we have
readily for $|\beta |>1$ 
\begin{equation}
\begin{split}
\big|\big(P_{\rho }\omega _{2}^{\ast }+i\beta P_{\rho }\omega _{1}^{\ast
},\omega _{1}\big)_{L^{2}(\Omega )}\big|& =\big|\big(\omega _{2}^{\ast
}+i\beta \omega _{1}^{\ast },\omega _{1}\big)_{L^{2}(\Omega )}+\rho \big(%
\nabla \lbrack \omega _{2}^{\ast }+i\beta \omega _{1}^{\ast }],\nabla \omega
_{1}\big)_{L^{2}(\Omega )}\big| \\[1mm]
& \leq C_{\rho }\,|\beta |\,\big\|\nabla \omega _{1}\big\|_{L^{2}(\Omega )}%
\big(\big\|\nabla \omega _{1}^{\ast }\big\|_{L^{2}(\Omega )}+\big\|\nabla
\omega _{2}^{\ast }\big\|_{L^{2}(\Omega )}\big) \\
& \leq 2\,\epsilon ^{2}\left\Vert \left[ 
\begin{array}{c}
\omega _{1} \\ 
\omega _{2} \\ 
\mu%
\end{array}%
\right] \right\Vert _{\mathbf{H}_{\rho }}^{2}+C_{\epsilon }|\beta
|^{2}\left\Vert \left[ 
\begin{array}{c}
\omega _{1}\ast \\ 
\omega _{2}^{\ast } \\ 
\mu ^{\ast }%
\end{array}%
\right] \right\Vert _{\mathbf{H}_{\rho }}^{2}\,,
\end{split}
\label{est6}
\end{equation}%
after reusing $|ab|\leq \epsilon a^{2}+C_{\epsilon }b^{2}$.

Applying \eqref{est4.5}, \eqref{est5}, and \eqref{est6} to the right hand
side of \eqref{mech}, and using the fact that $\omega _{1}$ satisfies hinged
boundary conditions, we then have 
\begin{equation}
\begin{split}
\Vert \Delta \omega _{1}\Vert _{L^{2}(\Omega )}^{2}& =\big(\Delta ^{2}\omega
_{1},\omega _{1}\big)_{L^{2}(\Omega )} \\
& \leq \epsilon \,\big(C_{\rho }+1+2\epsilon \big)\left\Vert \left[ 
\begin{array}{c}
w_{1} \\ 
w_{2} \\ 
u%
\end{array}%
\right] \right\Vert _{\mathbf{H}_{\rho }}^{2}+C_{\epsilon }\left\vert \beta
\right\vert ^{2}\left\Vert \left[ 
\begin{array}{c}
w_{1}^{\ast } \\ 
w_{2}^{\ast } \\ 
u^{\ast }%
\end{array}%
\right] \right\Vert _{\mathbf{H}_{\rho }}^{2}\,.
\end{split}
\label{pent}
\end{equation}

\medskip \noindent \emph{Step 4. (Control of the mechanical velocity)} Via
the resolvent relation $\omega_2=i\beta \omega_1-\omega_1^*$ we have 
\begin{equation*}
\|\omega_2\|_{H^1(\Omega )} \le \|\beta
\omega_1\|_{H^1(\Omega)}+\|\omega_1^*\|_{H^1(\Omega)} \le C\,\|\beta \nabla
\omega_1\|_{L^2(\Omega)}+\|\omega_1^*\|_{H^1(\Omega)}\,,
\end{equation*}
after again using Poincar\'e Inequality. Applying \eqref{est4} once more, we
attain 
\begin{equation}  \label{est8}
\|\omega_2\|_{H^1(\Omega )}^2\le \epsilon\,C\left\| \left[ 
\begin{array}{c}
\omega_1 \\ 
\omega_2 \\ 
\mu%
\end{array}
\right] \right\|_{\mathbf{H}_{\rho }}^2+C_{\epsilon }\,|\beta|^2\left\| %
\left[ 
\begin{array}{c}
\omega_1^* \\ 
\omega_2^* \\ 
\mu^*%
\end{array}
\right] \right\|_{\mathbf{H}_{\rho }}^2\,.
\end{equation}

\smallskip To finish the proof of Theorem~\ref{t:main}, we collect %
\eqref{dissi}, \eqref{pent} and \eqref{est8}. This gives the following
conclusion: the solution of the resolvent equation satisfies, for $|\beta|
>1 $, the estimate 
\begin{equation*}
\left\|\left[ 
\begin{array}{c}
\omega_1 \\ 
\omega_2 \\ 
\mu%
\end{array}
\right] \right\|_{\mathbf{H}_{\rho }}^2\le \epsilon\,C\left\|\left[ 
\begin{array}{c}
\omega_1 \\ 
\omega_2 \\ 
\mu%
\end{array}
\right] \right\|_{\mathbf{H}_{\rho }}^2+C_{\epsilon }\,|\beta|^2\left\|\left[
\begin{array}{c}
\omega_1^* \\ 
\omega_2^* \\ 
\mu^*%
\end{array}
\right] \right\|_{\mathbf{H}_{\rho}}^2\,,
\end{equation*}
which gives the estimate \eqref{object}, for $\epsilon >0$ small enough.
This concludes the proof of Theorem~\ref{t:main}. \qed
%\end{proof}

% THE END

\end{document}